\documentclass{amsart}

\theoremstyle{theorem}
\newtheorem{theorem}{Theorem}

\theoremstyle{definition}
\newtheorem*{definition}{Definition}
\newtheorem*{remark}{Remark}
\newtheorem{lemma}{Lemma}

\DeclareMathOperator{\tr}{tr}
\DeclareMathOperator{\vc}{vec}

\begin{document}

\title{A Short Note on Kronecker Square Roots}
\author{Yorick Hardy}
\address{
Department of Mathematical Sciences,
University of South Africa,
Johannesburg, South Africa
}
\email{hardyy@unisa.ac.za}

\maketitle

\begin{abstract}
 The results of [I. Ojeda, Amer. Math. Monthly, 122, pp 60--64]
 provides a characterization of Kronecker square roots of matrices
 in terms of the symmetry and rank of the block vec matrix (rearrangement matrix).
 In this short note we reformulate the characterization in terms of rank only
 by considering an alternative to the block vec matrix,
 provided that the characteristic of the underlying field is not equal to 2.
\end{abstract}

\bigskip

Let $\otimes$ denote the Kronecker product and $\bigotimes^k A$ the $k$-th
Kronecker power of a matrix $A$. An $m\times n$ matrix $A$ is said to be
a $k$-th Kronecker root of an $m^k\times n^k$ matrix $M$ if $M=\bigotimes^r A$.

Ojeda introduced the notion of the block vec matrix \cite{ojeda15a} to
characterize Kronecker square roots and to describe a simple procedure to
compute an Kronecker square root for real and complex matrices. In particular,
the rearrangement matrix in \cite{vanloan93a} plays a central role.
Let $A$ be an $m\times n$ matrix and $B$ be an $s\times t$ matrix.
The rearrangement operator $R_{m\times n}$ \cite{vanloan93a} is defined by
\begin{equation*}
 R_{m\times n}(A\otimes B) = \vc(A)\vc(B)^T,
\end{equation*}
and linear extension.
A characterization of Kronecker square roots is given in \cite{ojeda15a}:

\begin{theorem}[{\cite[Corollary 1]{ojeda15a}}]
 \label{thm:o1}%
 If $M$ is a non-zero $m^2\times n^2$ matrix and $A$ is an $m\times n$ matrix then
 \begin{enumerate}
  \item $M=A\otimes A$ if and only if $R_{m\times n}(M)=\vc(A)\vc(A)^T$,
  \item if $M=A\otimes A$, then $R_{m\times n}(M)$ is symmetric and has rank one.
 \end{enumerate}
\end{theorem}

This short note will show that (2) may be reformulated without reference
to symmetry (Theorem \ref{thm:nec}), provided that the underlying field is not of
characteristic 2. The reformulation of (1) and Theorem \ref{thm:o2} follows
trivially, as described at the end of this note.

\begin{theorem}[{\cite[Theorem 2]{ojeda15a}}]
 \label{thm:o2}%
 If $M$ is an $m^2\times n^2$ real or complex matrix
 such that $R_{m\times n}(M)$ is symmetric and has rank one, then
 \begin{enumerate}
  \item there exists an $m\times n$ complex matrix $A$ such that $M=A\otimes A$,
  \item if $A$ and $B$ are $m\times n$ matrices satisfying $M=A\otimes A=B\otimes B$,
        then $B=\pm A$,
  \item if $M$ is real then there exists an $m\times n$ real matrix $A$ such that $M=A\otimes A$ if and only
        if $\tr(R_{m\times n}(M))>0$.
 \end{enumerate}
\end{theorem}

In the proof of part (1) and part (3), Ojeda provides a straightforward procedure to determine
Kronecker square roots of matrices over the real or complex numbers. An analogous method holds
using the reformulation of Theorem \ref{thm:o1}.
First, we define the ``rearrangement'' operator $R_{m\times n}^{\Sigma}$, which plays a similar
role to $R_{m\times n}$.

\begin{definition}
 Let $A_1$, \ldots, $A_k$ be $m\times n$ matrices and $j\in\{1,\ldots,k\}$.
 Define the $j$-th rearrangement operator $R_{m\times n}^{(j)}$ by
 \begin{equation*}
  R_{m\times n}^{(j)}(A_1\otimes\cdots\otimes A_k)
   =\vc(A_j)\vc(A_1\otimes\cdots\otimes A_{j-1}\otimes A_{j+1}\otimes\cdots\otimes A_k)^T
 \end{equation*}
 and linear extension.
\end{definition}

It follows that for $k=2$ we have $R_{m\times n}(M) = R_{m\times n}^{(1)}(M) = (R_{m\times n}^{(2)}(M))^T$.
The $j$-th rearrangement operator $R_{m\times n}^{(j)}$ is bijective. Consequently
$M=\bigotimes^k A$ if and only if
\begin{equation*}
 R_{m\times n}^{(j)}(M)=\vc(A)\vc\left(\textstyle \bigotimes^{k-1} A\right)
\end{equation*}
for all $j\in\{1,\ldots,k\}$. In this case $R_{m\times n}^{(j)}(M)$ has rank one.

\begin{remark}
 The utility of $R_{m\times n}^{(j)}(M)$ is that it rearranges the entries of $M$
 in a configuration that is suitable for matrix rank analysis. Other rearrangements
 or unfoldings are equally applicable for this purpose.
\end{remark}

\begin{definition}
 Let $A_1$, \ldots, $A_k$ be $m\times n$ matrices. We define $R_{m\times n}^{\Sigma}$ by
 \begin{equation*}
  R_{m\times n}^{\Sigma}(A_1\otimes\cdots\otimes A_k)=\sum_{j=1}^k R_{m\times n}^{(j)}(A_1\otimes\cdots\otimes A_k)
 \end{equation*}
 and linear extension.
\end{definition}

Clearly $R_{m\times n}^{\Sigma}$ is not injective.

\begin{lemma}
 \label{lem:inj}%
 Suppose that the characteristic of the underlying field does not divide $k$.
 The restriction $R_{m\times n}^{\Sigma}:X\to R_{m\times n}^{(1)}(X)$ is bijective, where
 \begin{equation*}
  X=\left\{\,\alpha\textstyle\bigotimes^k A\,:\,\text{where $A$ is an $m\times n$ matrix and $\alpha$ a scalar}\,\right\}.
 \end{equation*}
\end{lemma}

\begin{proof}
 Let $\alpha\textstyle\bigotimes^k A\in X$ and $\beta\textstyle\bigotimes^k B\in X$, then
 \begin{equation*}
  R_{m\times n}^{\Sigma}\left(\alpha\textstyle\bigotimes^k A\right)
    = R_{m\times n}^{\Sigma}\left(\beta\textstyle\bigotimes^k B\right)
 \end{equation*}
 if and only if
 \begin{equation*}
  kR_{m\times n}^{(1)}\left(\alpha\textstyle\bigotimes^k A\right)
    = kR_{m\times n}^{(1)}\left(\beta\textstyle\bigotimes^k B\right)
 \end{equation*}
 and since $k\neq 0$ and $R_{m\times n}^{(1)}$ is bijective we have
 $\alpha\textstyle\bigotimes^k A=\beta\textstyle\bigotimes^k B$.
 Thus the restricted map $R_{m\times n}^{\Sigma}$ is injective.
 The restricted map $R_{m\times n}^{\Sigma}$ is surjective since
 \begin{equation*}
  R_{m\times n}^{\Sigma}\left(\frac{\alpha}k\textstyle\bigotimes^k A\right)
   = R_{m\times n}^{(1)}\left(\alpha\textstyle\bigotimes^k A\right).
 \end{equation*}
\end{proof}

\begin{remark}
 The $1^3\times 2^3$ matrix $M=\begin{pmatrix}1&-1&1&0&0&0&0&0\end{pmatrix}$
 is not a Kronecker cube, but $R_{1\times 2}^{\Sigma}(M)$ is a rank 1 matrix.
 Of course, $M\notin X$ in this example.
\end{remark}

With the above definitions and lemma, we are ready to extend Theorem \ref{thm:o1} for
higher order Kronecker roots.

\begin{theorem}
 \label{thm:nec}%
 Suppose that the characteristic of the underlying field does not divide $k$.
 If $M$ is a non-zero $m^k\times n^k$ matrix and $A$ is an $m\times n$ matrix
 such that $M=\bigotimes^k A$, then $R_{m\times n}^{\Sigma}(M)$ has rank one.
\end{theorem}

\begin{proof}
 Since for $M=\bigotimes^k A$ we have that
 \begin{equation*}
  R_{m\times n}^{\Sigma}(M)=k\vc(A)\vc\left(\textstyle \bigotimes^{k-1} A\right),
 \end{equation*}
 and since the characteristic of the underlying field does not divide $k$, $k\neq 0$
 and it follows that $R_{m\times n}^{\Sigma}(M)$ has rank one.
\end{proof}

The connection between Theorem \ref{thm:o1} part (2) and Theorem \ref{thm:nec} is as
follows.

\begin{theorem}
 Let $k=2$ and let $M$ be a non-zero $m^2\times n^2$ matrix over a field with characteristic
 not equal to 2. Then $R_{m\times n}(M)$ is symmetric and has rank one
 if and only if $R_{m\times n}^{\Sigma}(M)$ has rank one.
\end{theorem}

\begin{proof}
 Recall that, for $k=2$, $R_{m\times n}^{(2)}(M)=R_{m\times n}^{(2)}(M)^T=R_{m\times n}(M)$.
 If $R_{m\times n}(M)$ is symmetric and has rank one, then
 \begin{equation*}
  R_{m\times n}^{\Sigma}(M) = R_{m\times n}^{(1)}(M) + R_{m\times n}^{(2)}(M) = 2R_{m\times n}(M)
 \end{equation*}
 which has rank one. For the converse, suppose $R_{m\times n}^{\Sigma}(M)$ has rank one.
 Let
 \begin{equation*}
  M=\sum_{j=1}^p A_j\otimes B_j
 \end{equation*}
 be a tensor decomposition of $M$, where the matrices $A_j$ and $B_j$ are $m\times n$ matrices. Then
 \begin{equation*}
  R_{m\times n}^{\Sigma}(M)=\sum_{j=1}^p \left[\vc(A_j)\vc(B_j)^T + \vc(B_j)\vc(A_j)^T\right].
 \end{equation*}
 Since $R_{m\times n}^{\Sigma}(M)$ has rank one, we find that the matrices $A_j$ and $B_j$ are all scalar
 multiples of the same non-zero matrix $A_*$. Thus
 \begin{equation*}
  R_{m\times n}(M) = \sum_{j=1}^p \vc(A_j)\vc(B_j)^T = c\vc(A_*)\vc(A_*)^T,
 \end{equation*}
 for some scalar $c$. Clearly $c\neq0$, since if $R_{m\times n}(M)$ is zero, then
 $M$ is also zero and consequently $R_{m\times n}^{\Sigma}(M)$ is also zero (i.e.
 does not have rank 1).
 Consequently $R_{m\times n}(M) = c\vc(A_*)\vc(A_*)^T$ is symmetric and has rank one.
\end{proof}

Theorem \ref{thm:o1} part (1) and Theorem \ref{thm:o2} parts (1) and (3) and their proofs
in \cite{ojeda15a} may be trivially reformulated using $R_{m\times n}(M)=\frac12R_{m\times n}^{\Sigma}(M)$
when the underlying field does not have characteristic equal to 2.


\begin{thebibliography}{1}

\bibitem{ojeda15a}
Ignacio Ojeda, \emph{{K}ronecker square roots and the block vec matrix}, The
  American Mathematical Monthly \textbf{122} (2015), no.~1, 60--64.

\bibitem{vanloan93a}
Charles~F. van Loan and Nikos Pitsianis, \emph{Approximation with {K}ronecker
  products}, Linear Algebra for Large Scale and Real-Time Applications (Marc~S.
  Moonen, Gene~H. Golub, and Bart L.~R. De~Moor, eds.), Kluwer Publications,
  1993, pp.~293--314.

\end{thebibliography}
\providecommand{\bysame}{\leavevmode\hbox to3em{\hrulefill}\thinspace}
\providecommand{\MR}{\relax\ifhmode\unskip\space\fi MR }
\providecommand{\MRhref}[2]{%
  \href{http://www.ams.org/mathscinet-getitem?mr=#1}{#2}
}
\providecommand{\href}[2]{#2}

\end{document}